\documentclass[final,leqno,a4paper]{amsart}
\usepackage[dvips]{graphicx}

\usepackage{amsfonts,amsmath,amscd,amssymb,latexsym,verbatim,pb-diagram,cite}

\usepackage[dvips]{graphicx}
\usepackage[english]{babel}

\usepackage{amsthm}

\usepackage{amsfonts,amsmath,amscd,amssymb,latexsym,verbatim}
\numberwithin{equation}{section}

\theoremstyle{definition}
\newtheorem{definition}{Definition}[section]
\newtheorem{remark}[definition]{Remark}
\newtheorem{example}[definition]{Example}

 \theoremstyle{plain}

\newtheorem{theorem}[definition]{Theorem}
\newtheorem{lemma}[definition]{Lemma}
\newtheorem{proposition}[definition]{Proposition}
\newtheorem{corollary}[definition]{Corollary}

\usepackage[notref,notcite]{showkeys}

\DeclareMathOperator{\Id}{Id}
\DeclareMathOperator{\diam}{diam}
\DeclareMathOperator{\card}{card}

\renewcommand{\mathbb}{\mathsf}

\begin{document}



\title[Products of metric spaces...]{Products of metric spaces,\\
covering numbers, packing numbers\\ and characterizations of
ultrametric spaces}

\author{Oleksiy Dovgoshey and Olli Martio}

\address{Institute of Applied Mathematics and Mechanics of NASU,
 R.~Luxemburg str. 74, Donetsk 83114,
Ukraine}

\email{aleksdov@mail.ru}

\address{Department of Mathematics and Statistics, University of Helsinki,
 P.O. Box 68 FI-00014 University of
Helsinki, Finland}

\email{olli.martio@helsinki.fi}

\begin{abstract} We describe some Cartesian products of metric spaces and
find conditions under which products of ultrametric spaces are
ultrametric.
\end{abstract}



\maketitle

\section{Introduction}

Let $(X,d)$ be a metric space. The closed balls with a center
$c\in X$ and radius $r,\ 0<r<\infty$, are denoted by
\begin{equation*}
B(c,r)=B_d(c,r)=\{x\in X:d(x,c)\leq r\}.
\end{equation*}
Let $W$ be a subset of $X$ and let $\varepsilon>0$. A set
$C\subseteq X$ is an \textit{$\varepsilon$-net} for $W$ if
\begin{equation*}
W\subseteq\bigcup_{c\in C}B(c,\varepsilon).
\end{equation*}
A set $W\subseteq X$ is called {\it totally bounded (or
precompact)} if for every $\varepsilon>0$ there is a finite
$\varepsilon$-net for $W$.  The {\it covering number} of a totally
bounded set $W\subseteq X$ is the smallest cardinality of subsets
of $W$ which are $\varepsilon$-nets for $W$. A set $A\subseteq X$
is called {\it$\varepsilon$-distinguishable} if
$d(x,y)>\varepsilon$ for every distinct points $x,y\in A$,
\cite{KolTi}. The {\it packing number} of a precompact set
$W\subseteq X$ is the maximal cardinality of the
$\varepsilon$-distinguishable sets $A\subseteq W$.

We denote by $\mathcal N_\varepsilon(W)$ and by $\mathcal
M_\varepsilon(W)$ the covering number and, respectively, the
packing number of a totally bounded set $W\subseteq X$. These
quantities have been invented by Kolmogorov \cite{Kol1} in order
to classify compact metric sets. Note that the function
$\log_2\mathcal N_\varepsilon(W)$ is the so-called {\it metric
entropy}  and it has been widely applied in approximation theory,
geometric functional analysis, probability theory and complexity
theory, see, for example, \cite{KolTi,Lo2, CarSt,EdTri}.

A main general fact about packing and covering numbers is the
simple double inequality
\begin{equation}\label{eq1.1}
\mathcal M_{2\varepsilon}(W)\leq\mathcal N_\varepsilon(W)\leq
\mathcal M_\varepsilon(W).
\end{equation}
In the second section of this paper we consider some transfinite
generalizations of covering numbers and packing ones and obtain a
more exact version of inequality \eqref{eq1.1}, see Lemma
\ref{l:3.2}. It implies the characterization of ultrametric spaces
as spaces for which packing numbers equal covering numbers. In the
third and fourth sections we introduce some ``natural'' metrics on
the products of metric spaces and discuss conditions under which
the products of ultrametric spaces are ultrametric.

\section{The equality between covering numbers\\ and packing numbers}

Let $(X,d)$ be a metric space. Denote by $t_0=t_0(d)$ the supremum
of positive numbers $t$ for which the function
$(x,y)\longmapsto(d(x,y))^t$ is a metric on $X$. This quantity has
the following characterization, see \cite{DoMa2}.

\begin{lemma}
\label{l:1.1} Let $x,y$ and $z$ be points in a metric space
$(X,d)$. If the inequality
\begin{equation}  \label{1.1}
\max\{d(x,z), d(z,y)\}<d(x,y)
\end{equation}
holds, then there exists a unique solution $s_0\in[1,\infty[$ of
the equation
\begin{equation}  \label{1.2}
(d(x,y))^s=(d(x,z))^s+(d(z,y))^s.
\end{equation}
\end{lemma}

For points $x,y$ and $z$ in $X$ write
\begin{equation}  \label{1.3}
s(x,y,z):=%
\begin{cases}
s_0 & \text{if (\ref{1.1}) holds } \\
+\infty & \text{otherwise}%
\end{cases}%
\end{equation}
where $s_0$ is the unique root of equation \eqref{1.2}.

\begin{proposition}
\label{p:1.2'} The equality
\begin{equation*}
t_0(d)=\inf\{s(x,y,z):x,y,z\in X\}
\end{equation*}
holds in every metric space $(X,d)$.
\end{proposition}

\begin{remark}
\label{r:1.3}  A point $%
z $ in a metric space $(X,d)$ \textit{lies between} two distinct
points $x$ and $y$ if $d(x,z)+d(z,y)=d(x,y)$ and $x\ne z\ne y$,
see \cite[p.~55]{Pa}. Now $t_0=t_0(d)$ can be called the
\textit{betweenness exponent} of the space $(X,d)$.
\end{remark}
Recall that a metric space $(X,d)$ is \textit{ultrametric} if the
metric $d$ satisfies the \textit{ultra-triangle inequality}
$d(x,y)\leq \max\{d(x,z), d(z,y)\}$ for all $x,y,z\in X$. In this
case $d$ is called an \textit{ultrametric}. Since \eqref{1.1}
never holds in an ultrametric space $(X,d)$ we have
$t_0(d)=\infty$ in this case. In fact, $(X,d)$ is ultrametric if
and only if $t_0(d)=\infty$.

\begin{lemma}
\label{l:1.4} Let $B(a,r)$ be a closed ball in a metric space
$(X,d)$. Then we have the inequality
\begin{equation}  \label{1.4}
\diam(B(a,r))\leq2^{\frac1{t_0}}r
\end{equation}
where $t_0=t_0(d)$ is the betweenness exponent of $(X,d)$.
\end{lemma}

\begin{proof}
If $t_0(d)=\infty$, then $(X,d)$ is ultrametric and $\diam
B(a,r)\leq r$ for every ball $B(a,r)$, see, for example,
\cite[p.~43]{Ed}. In the case $t_0(d)<\infty$ the function
$(x,y)\longmapsto (d(x,y))^{t_0}$ is a metric. Hence, by the
triangle inequality, we have
$$
d^{t_0}(x,y)\leq d^{t_0}(x,a)+ d^{t_0}(y,a)\leq 2r^{t_0}
$$
for all $x,y\in B(a,r)$. The last inequality implies \eqref{1.4}.
\end{proof}
There is the possibility of more refined classification of
nonprecompact metric spaces by means of an extension of the range
of values of the functions $\mathcal{N}_\varepsilon$ and $\mathcal
M_\varepsilon$ to transfinite cardinal numbers.

Let $W$ and $A$ be subsets of $X$. Define
\begin{equation}\label{eq2.5}
\hat{\mathcal N}_\varepsilon^A(W):=\min\{\card(C): C\text{ is an
$\varepsilon$-net for $W$ and }C\subseteq A\}.
\end{equation}
Moreover for the sake of simplicity, write $\hat{\mathcal
N}_\varepsilon(W):=\hat{\mathcal N}_\varepsilon^W(W)$.

For convenience we introduce an additional definition.
\begin{definition}
\label{d:3.1} A set $A$ is {\it maximal
$\varepsilon$-distinguish\-able} with respect to $W$ if $A$ is
$\varepsilon$-distin\-guish\-able, $A\subseteq W$ and for every
$\varepsilon$-distinguishable $B\subseteq W$ the inclusion
$A\subseteq B$ implies the equality $A=B$.
\end{definition}

Write $\hat{\mathcal{M}}_\varepsilon(W)$ for the smallest power of
maximal $\varepsilon$-distinguishable sets $A\subseteq W$ and
 define the quantity $%
{\mathcal{M}}^*_\varepsilon(W)$ as the smallest cardinal number
which is
greater than or equal to $\card(A)$ for every $\varepsilon$-distinguishable $%
A\subseteq W$. It is clear that
$$
\mathcal M_\varepsilon^*(W)=\mathcal
M_\varepsilon(W)\quad\text{and}\quad \hat{\mathcal
N}_\varepsilon(W)=\mathcal N_\varepsilon(W)
$$
for every precompact $W$.
\begin{lemma}
\label{l:3.2} Let $W$ be a set in a metric space $(X,d)$. Then for every $%
\varepsilon>0$ we have the following inequalities
\begin{equation}  \label{3.1}
\mathcal{M}^*_{2^{\frac1{t_0}}\varepsilon}(W)\leq \hat{\mathcal{N}}%
_\varepsilon^X(W)\leq\hat{\mathcal{N}}_\varepsilon(W)\leq\hat{\mathcal{M}}%
_\varepsilon(W)\leq{\mathcal{M}}^*_\varepsilon(W)
\end{equation}
where $t_0$ is the betweenness exponent of $X$.
\end{lemma}

\begin{proof}
The first inequality from the right is immediate. For the proof of
the second one note that every maximal
$\varepsilon$-distinguishable set $A\subseteq W$ is an
$\varepsilon$-net for $W$. The inequality $ \hat{\mathcal
N}_\varepsilon^X(W) \leq \hat{\mathcal N}_\varepsilon(W) $ is
clear from the definitions. To prove the first inequality from the
left it suffices to show $ \card(A)\leq\hat{\mathcal
N}_\varepsilon^X(W) $ for every
$2^{\frac1{t_0}}\varepsilon$-distinguishable set $A \subseteq W$.
Let $\{x_i:i\in I\}$ be an $\varepsilon$-net for $W$ with
$\card(I)=\hat{\mathcal N}_\varepsilon^X(W)$. Suppose that there
exists a $2^{\frac1{t_0}}\varepsilon$-distinguishable set
$A_0\subseteq W$ for which
$$
\card(A_0)>\card(I).
$$
This inequality and the inclusion
$$
A_0\subseteq\bigcup_{i\in I}B(x_i,\varepsilon)
$$
imply that there exists a ball $B(x_i,\varepsilon)$ which contains
at least two distinct points $y_i,z_i\in A_0$. (In the opposite
case $A_0$ and some subset of $I$ have the same cardinality.)
Lemma \ref{l:1.4} implies that
$$d(y_i,z_i)\leq2^{\frac{1}{t_{0}}}\varepsilon.$$
This contradicts the assumption that $A_0$ is
$2^{\frac{1}{t_{0}}}\varepsilon\text{-distinguishable}$.

\end{proof}

\begin{corollary}
Let $X$ be a nonprecompact  metric space. Then for some
$\varepsilon_0>0$ there is an $\varepsilon_0$-distinguishable,
countable infinite set $A\subseteq X$.
\end{corollary}

\begin{example}
\label{e:3.3} Let $X$ be a set of a power $\alpha>2$ and let $a$
be an element of $X$. For every two distinct $x,y\in X$ write
\begin{equation*}
d(x,y)=
\begin{cases}
2^{\frac1{t}} & \text{if } x\ne a\ne y \\
1 & \text{otherwise}%
\end{cases}%
\end{equation*}
where $t\in[1,\infty[$ and put $d(x,y)=0$ if $x=y$. Proposition
\ref{p:1.2'} implies that the metric space $(X,d)$ has the
betweenness exponent $t_0(d)=t$. If we define a set $W$ as
$W=X\setminus\{a\}$, then
\begin{equation*}
\hat{\mathcal{N}}_\varepsilon(W)=\hat{\mathcal{M}}_\varepsilon(W)=
\mathcal{M}^*_\varepsilon(W)= \card(W)
\end{equation*}
but
\begin{equation*}
\hat{\mathcal{N}}^{X}_\varepsilon(W)=\mathcal{M}^*_{2^{\frac1{t}}
\varepsilon}(W)=1=\hat{\mathcal{M}}_\varepsilon(X)
\end{equation*}
for every $\varepsilon\in\;]1,2^{\frac1{t}}[$.
\end{example}

\begin{theorem}
\label{t:1.6} Let $(X,d)$ be a metric space. The following
statements are equivalent.

\begin{itemize}
\item[$(i)$] The space $X$ is ultrametric.

\item[$(ii)$] For every $W\subseteq X$  the equalities
\begin{equation}\label{1.7}
\mathcal{M}_{\varepsilon }^{\ast }(W)=\hat{\mathcal{N}}_{\varepsilon }^{X}(W)=\hat{\mathcal{N}}%
_{\varepsilon }(W)=\hat{\mathcal{M}}_{\varepsilon }(W)
\end{equation}%
hold for all $\varepsilon >0.$

\item[$(iii)$] For every compact $W\subseteq X$ and every
$\varepsilon>0$ we have the equality
\begin{equation*}
\mathcal{N}_\varepsilon(W)=\mathcal{M}_\varepsilon(W).
\end{equation*}
\end{itemize}
\end{theorem}

\begin{proof}
Since $t_0(d)=\infty$ holds if $d$ is an ultrametric, inequalities
\eqref{3.1} impliy \eqref{1.7} for ultrametric spaces. The
implication (ii)$\Rightarrow $(iii) is trivial. If $(X,d)$ is not
an ultrametric space, then there are points $a,b,c\in X$ such that
\begin{equation}\label{1.8}
d(a,b)>\max\{d(a,c), d(b,c)\}.
\end{equation}
Write $\varepsilon:=\max\{d(a,c), d(b,c)\}$. It follows from
\eqref{1.8} that $\mathcal M_\varepsilon(\{a,b,c\})\geq2$.
Moreover, since $B(c,\varepsilon)\supseteq\{a,b,c\}$, we see that
$\mathcal N_\varepsilon(\{a,b,c\})\leq 1$. Hence $\mathcal
N_\varepsilon(\{a,b,c\})\ne\mathcal M_\varepsilon(\{a,b,c\})$.
\end{proof}
Consider now equalities \eqref{1.7} for non ultrametric spaces.

Recall that a cardinal number $\alpha$ is the \emph{density} of a
metric space $X$ if
\begin{equation*}
\alpha =\underset{A}{\text{min}}(\text{card}(A))
\end{equation*}
where the minimum is taken over the family of all  dense sets $A
\subseteq X$. For the density of $X$ we use the symbol den$X$. For
convenience we repeat some definitions related to the confinality
of the cardinals, see, for example \cite{Ro}. We understand the
ordinal numbers as some special well-ordered sets
$\alpha,\beta,... $ for which the statements: \\
-$\alpha$ is similar to an  initial segment of
$\beta$ and $\alpha\neq\beta$,   $\alpha\prec\beta$; \\
-$\alpha$ is proper subset of $\beta$, $\alpha\subsetneq\beta$;\\
-$\alpha$ belongs to $\beta$, $\alpha\in\beta$ \\
are equivalent. An ordinal number $\beta$ is an \emph{initial}
ordinal if for all ordinals $\alpha$ we have the implication
\begin{equation*}
(\alpha\prec\beta)\Rightarrow(|\alpha|\lneq|\beta|)
\end{equation*}
where $|\alpha|$ and $|\beta|$ are corresponding cardinality of
$\alpha$ and $\beta$. By cardinal numbers we mean initial
ordinals. An ordinal number $\alpha$ is \emph{confinal} in an
ordinal $\beta$ if there is an one-to-one increasing mapping
$f:\alpha\rightarrow\beta$ such that for every ordinal
$\gamma\in\beta$ there exists an ordinal $\delta\in\alpha$ with
\begin{equation*}
\gamma\prec f(\delta)\qquad \text{or}\qquad  \gamma=f(\delta).
\end{equation*}
The \emph{confinality} of an ordinal $\beta$ is the least ordinal
$\alpha$ with $\alpha$ confinal in $\beta$. We write cf$(\beta)$
for the confinality of $\beta$. If $\alpha$ is the confinality for
some $\beta$, then $\alpha$ is a cardinal,\cite[p.91]{Ro}.
\begin{theorem}\label{t:1.12}
Let $W$ be a subset of a metric space $X$. Suppose that $den(W)$
is a cardinal of an uncountable confinality. Then there is
$\varepsilon _{0}>0$ such that the equalities
\begin{equation}
\hat{\mathcal{N}}_{\varepsilon }^{X}(W)=\hat{\mathcal{N}}_{\varepsilon }(W)=\hat{\mathcal{M}}%
_{\varepsilon }(W)=\hat{\mathcal{M}}_{\varepsilon }^{\ast
}(W)=\text{den}(W) \label{3.2}
\end{equation}%
hold for all $\varepsilon \in \left] 0,\varepsilon _{0}\right[$.
\end{theorem}

Write, as usual, $\aleph_0$ for  $\text{card}(\mathbb{N})$ and
$\mathfrak{c}=2^{\aleph_0}$=card($\mathbb{R}$).

\begin{corollary}
Let $W$ be a subset of a metric space $X$. If \
$den(W)=\mathfrak{c}$, then there is $\varepsilon _{0}>0$ such
that the equalities
\begin{equation}
\hat{\mathcal{N}}_{\varepsilon
}^{X}(W)=\hat{\mathcal{N}}_{\varepsilon}(W)=\hat{{\mathcal{M}}}
_{\varepsilon}(W)=\hat{{\mathcal{M}}}_{\varepsilon}^{\ast}(W)=\mathfrak{c}
\label{3.3}
\end{equation}%
hold for all $\varepsilon \in \left] 0,\varepsilon _{0}\right[$.
\end{corollary}

\begin{proof}
Since for each infinite cardinal $\gamma$ we have $ \gamma\prec
cf(2^{\gamma }), $ see \cite[Theorem 44,p.93]{Ro}, $\mathfrak{c}$
has an uncountable confinality.
\end{proof}
\begin{corollary}\label{c:2.12}
Let $(X,\tau)$ be a metrizable topological space, let $W\subseteq
X$ be a set such that {\rm den}$\, (W)$ is a cardinal of an
uncountable confinality and let $D$ be a finite family of metrics
$d$ each of which induces the topology $\tau$ on $X$. Then there
is $\varepsilon_0>0$ such that for all
$\varepsilon\in\;]0,\varepsilon_0[$ the values $\hat{\mathcal
N}_\varepsilon(W)$, $\hat{\mathcal N}^X_\varepsilon(W)$,
$\hat{\mathcal M}_\varepsilon(W)$ and $\mathcal
M^*_\varepsilon(W)$ do not depend on the choice of $d\in D$.
\end{corollary}

\begin{proof}[\it Proof of Theorem \ref{t:1.12}.] The definitions of cardinal numbers
$\hat{\mathcal{N}}_{\varepsilon }(W)$ and den$(W)$ imply that the inequality%
\[
\hat{\mathcal{N}}_{\varepsilon }^{X}(W)\leq
\hat{\mathcal{N}}_{\varepsilon}(W)\leq \text{den}(W)
\]%
holds for all $\varepsilon >0.$ Hence, by (\ref{3.1}), we have

\[
\hat{\mathcal{M}}_{\varepsilon }(W)\leq \mathcal{M}_{\varepsilon
}^{\ast }(W)\leq \text{den}(W)
\]%
if $\varepsilon >0.$ Moreover, if there is $\varepsilon _{0}>0$ $\
$such that
\begin{equation}
\text{den}(W)\leq \hat{\mathcal{N}}_{\varepsilon _{0}}(W),
\label{3.4}
\end{equation}%
then the last inequality and (\ref{3.1}) imply%
\[
\hat{\mathcal{N}}_{\varepsilon _{0}}^{X}(W)\geq
\mathcal{M}_{2^{-\frac{1}{t_{0}}}\varepsilon _{0}}^{\ast }(W)\geq
\hat{\mathcal{M}}_{2^{-\frac{1}{t_{0}}}\varepsilon _{0}}(W)\geq
\hat{\mathcal{N}}_{2^{-\frac{1}{t_{0}}}\varepsilon _{0}}(W)\geq
\text{den}(W).
\]
Therefore, it is sufficient to show (\ref{3.4}) with some
$\varepsilon _{0}>0$ .

If $D$ is a dense subset of $W$, then for every $k\in \left]
0,1\right[ $ and all $\varepsilon >0$ we have the double
inequality
\begin{equation}
\hat{\mathcal{N}}_{k\varepsilon }(W)\geq
\hat{\mathcal{N}}_{\varepsilon
}(D)\geq\hat{\mathcal{N}}_\frac{\varepsilon }{k}(W). \label{3.6}
\end{equation}
Indeed, if $C=\left\{ c_{i}:i\in I\right\}$ is a $k \varepsilon $%
-net for $W$ with $ \text{card}(C)=\hat{\mathcal{N}}_{k\varepsilon
}(W), $
then the density of $D$ in $W$ implies that for every $c_{i}\in C$ there is $%
b_{i}\in D$ such that
$
B(b_{i},\varepsilon )\supseteq B(c_{i},k\varepsilon).
$
Hence we have
\[
D\subseteq W\subseteq \bigcup\limits_{i\in I}B(c_{i},k\varepsilon
)\subseteq \bigcup\limits_{i\in I}B(b_{i},\varepsilon),
\]%
i.e., $\{b_{i}:i\in I\}$  is an $\varepsilon$-net for $D$, so the
first inequality in \eqref{3.6} is proved. Similarly, if
$P=\left\{ p_{i}:i\in I\right\}$ is an \emph{$\varepsilon$-net}
for $D$ with $ \text{card}(P)=\hat{\mathcal{N}}_{\varepsilon }(D),
$ then for every $x\in W$ there is $p_{i}\in P$ such that
\begin{equation*}
x\in B(p_{i},\frac{\varepsilon }{k}).
\end{equation*}%
Hence $P$ is an \emph{$\frac{\varepsilon}{k}$-net} for $W$, that
implies the second inequality in \eqref{3.6}.

Let $D$ be a dense subset of $W$ such that
\begin{equation}\label{3.61}
\text{card}(D)=\text{den}(W).
\end{equation}

Consider a sequence of positive numbers $\varepsilon
_{1},\varepsilon _{2},...$ with $\text{lim}_{i\rightarrow \infty
}\varepsilon _{i}=0$. Suppose that  a set $D_i$ is an
\emph{$\varepsilon_{i}$-net} for $D$ with $D_{i}\subseteq D$ and
with
\begin{equation}
\text{card}(D_{i})=\hat{\mathcal{N}}_{\varepsilon _{i}}(D)
\label{3.7}
\end{equation}for every $i\in\mathbb N$.
 The set
\begin{equation}
\tilde{D}:=\bigcup\limits_{i=1}^{\infty}D_{i}  \label{3.8}
\end{equation}%
is a dense subset of $W$ and $\tilde{D}\subseteq D$. Hence, by
(\ref{3.61}),
$
\text{card}(\tilde{D})=\text{den}(W).
$
Suppose also that the inequality%
\begin{equation}
\text{card}(D_{i})\lneq\text{den}(W)  \label{3.9}
\end{equation}%
holds for each $D_{i}$. Let $\gamma $ be an initial ordinal such
that $ \left\vert \gamma \right\vert =\text{card}(\tilde{D}) $ and
let $f:\gamma \rightarrow \tilde{D}$ be a bijection. Inequality
\eqref{3.9} implies that for every ordinal $ \alpha
_{i}:=f^{-1}(D_{i}) $ there is an ordinal $\beta _{i}\in \gamma$
such that  $\alpha _{i}$ is similar to an initial segment of
$\beta_{i}$ and $\alpha_i\neq\beta_i$.

From this and (\ref{3.8}) it follows that $\aleph_0$ is confinal
in the ordinal number $\text{den}(W)$, contrary to the supposition
of the theorem. Thus there is $\varepsilon _{i_{0}}>0$ such that
$
\text{card}(D_{i_{0}})=\text{den}(W).
$
This equality and (\ref{3.6}) imply (\ref{3.4}) with $\varepsilon
_{0}=k\varepsilon _{i_{0}}$.
\end{proof}

\section{Metrics on  products of metric spaces}

Let $(X,d_X)$ and $(Y,d_Y)$ be two metric spaces.

\begin{definition}\label{2:d1.1}
A metric $d$ on the product $X\times Y$ is said to be {\it
distance-increasing} if
\begin{equation}\label{2:1.1}
d((x_1,y_1),(x_2,y_2))\leq d((x_3,y_3),(x_4,y_4))
\end{equation}
whenever
\begin{equation}\label{2:1.2}
d_X(x_1,x_2)\leq d_X(x_3,x_4)\quad\text{and}\quad d_Y(y_1,y_2)\leq
d_Y(y_3,y_4);
\end{equation}
$d$ is {\it partial distance-preserving} if  we have the
equalities
\begin{equation}\label{2:1.3}
d((x_1,y),(x_2,y))=d_X(x_1,x_2)\quad\text{and}\quad
d((x,y_1),(x,y_2))=d_Y(y_1,y_2)
\end{equation}
for all $x,x_1,x_2\in X$ and $y,y_1,y_2\in Y$.
\end{definition}
\begin{remark}\label{2:r1.2}
When
\begin{equation}\label{2:1.4}
d_X(x_1,x_2)=d_X(x_3,x_4),\qquad d_Y(y_1,y_2)=d_Y(y_3,y_4),
\end{equation}
 we obtain from \eqref{2:1.1} and \eqref{2:1.2} that
\begin{equation}\label{2:1.5}
d((x_1,y_1),(x_2,y_2))=d((x_3,y_3),(x_4,y_4)),
\end{equation}
i.e., the distance-function $d:(X\times Y)\times(X\times
Y)\to\mathbb R^+$ depends only on ``partial'' distance-functions
$d_X$ and $d_Y$. Consequently, there is a mapping $F:D_X\times
D_Y\to\mathbb R^+$ with
\begin{equation}\label{eq2.6}
D_X:=\{d_X(x,y):x,y\in X\},\qquad D_Y:=\{d_Y(x,y):x,y\in Y\}
\end{equation}
such that the following diagram
\begin{equation}\label{2:1.6}
\begin{diagram}
\node{(X\times Y)\times(X\times Y)}\arrow[3]{e,t}{d}
                                   \arrow{s,l}{\Id} \node[3]{\mathbb
                                   R^+}\\
\node{(X\times X)\times(Y\times Y)}\arrow[3]{e,t}{d_X\otimes d_Y}
\node[3]{D_X\times D_Y}\arrow{n,r}{F}
\end{diagram}
\end{equation}
is commutative. Here $\Id$ is an identification mapping
$$
\Id((x_1,y_1),(x_2,y_2))=((x_1,x_2),(y_1,y_2))
$$
and $d_X\otimes d_y$ is the direct product of the partial distance
functions $d_X$ and $d_Y$,
$$
d_X\otimes d_Y((x_1,x_2),(y_1,y_2))=(d_X(x_1,x_2),d_Y(y_1,y_2)).
$$
\end{remark}
Diagram \eqref{2:1.6}  shows that we can find the metric
properties of the product $X\times Y$ using the corresponding ones
of the function $F$. This approach to the study of metric products
was originated at the paper of A.~Bernig, T.~Foertsch and
V.~Schroeder \cite{BFS}.
\begin{example}\label{2:e1.3}
For every $p\in[1,\infty]$ let $d_p$ be a metric on $X\times Y$
defined as
\begin{equation}\label{2:1.7}
d_p((x_1,y_1),(x_2,y_2))=((d_X(x_1,x_2))^p+(d_Y(y_1,y_2))^p)^{\frac1p}
\end{equation}
if $1\leq p<\infty$ and
\begin{equation}\label{2:1.8}
d_\infty((x_1,y_1),(x_2,y_2))=\max\{d_X(x_1,x_2),d_Y(y_1,y_2)\}
\end{equation}
if $p=\infty$.  It is clear that the metrics $d_p$ are
distance-increasing and partial distance-preserving for every
$p\in[1,\infty]$.
\end{example}
\begin{proposition}\label{2:p1.4}
Let $(X,d_X)$ and $(Y,d_Y)$ be two metric spaces and let $d$ be a
distance-increasing, partial distance-preserving metric on the
product $X\times Y$. Then the following double inequality holds
for all and $(x_i,y_i)\in X\times Y,\ i=1,2$,
\begin{equation}\label{2:1.9}
d_\infty((x_1,y_1),(x_2,y_2))\leq d((x_1,y_1),(x_2,y_2))\leq
d_1((x_1,y_1),(x_2,y_2))
\end{equation}
where metrics $d_\infty$ and $d_1$ are defined by \eqref{2:1.8}
and \eqref{2:1.7}, respectively.
\end{proposition}
\begin{proof}
To prove the first inequality in \eqref{2:1.9} we may assume that
\begin{equation}\label{2:1.10}
d_\infty((x_1,y_1),(x_2,y_2))=d_X(x_1,x_2).
\end{equation}
Since $d_Y(y_1,y_2)\geq0= d_Y(y_1,y_1)$ and $d$ is
distance-increasing,
$$
d((x_1,y_1),(x_2,y_2))\geq d((x_1,y_1),(x_2,y_1)).
$$
This inequality, the first equality in \eqref{2:1.3} and
\eqref{2:1.10} imply that
$$
d((x_1,y_1),(x_2,y_2))\geq
d_X(x_1,x_2)=d_\infty((x_1,y_1),(x_2,y_2)),
$$
i.e., the first inequality in \eqref{2:1.9} holds.

To prove the right hand side of \eqref{2:1.9} consider the
following triangle inequality for the metric $d$
\begin{equation}\label{3:12*}
d((x_1,y_1),(x_2,y_2))\leq
d((x_1,y_1),(x_1,y_2))+d((x_1,y_2),(x_2,y_2)).
\end{equation}
From this and \eqref{2:1.3} we obtain
$$
d((x_1,y_2),(x_2,y_2))\leq
d_X(x_1,x_2)+d_Y(y_1,y_2)=d_1((x_1,y_1),(x_2,y_2)),
$$
as required.
\end{proof}
Recall that there is a {\it natural topology} on the product
space, it is the coarsest topology for which the canonical
projections to the factors are continuous.
\begin{corollary}\label{2:c1.5}
Let $(X,d_X)$ and $(Y,d_Y)$ be two metric spaces. All
distance-increasing, partial distance-preserving metrics on the
product $X\times Y$ induce the natural topology on this product.
\end{corollary}

\begin{proof}
Let $d$ be a partial distance-preserving, distance-increasing
metric on $X\times Y$. Inequality \eqref{2:1.9} implies that
$d_\infty\leq d\leq2d_\infty$. Hence the spaces $(X\times
Y,d_\infty)$ and $(X\times Y,d)$ have the coinciding sets of
convergent sequences. Consequently these spaces have the same
topology. Moreover, it is well-known that $d_\infty$ induces the
natural topology on $X\times Y$. Therefore the topology of the
space $(X\times Y,d)$ also is natural.
\end{proof}

Proposition \ref{2:p1.4} admits a partial converse.
\begin{proposition}\label{2:p1.6}
Let $(X,d_X)$ and $(Y,d_Y)$ be two metric spaces. If $d$ is a
metric on $X\times Y$ such that double inequality \eqref{2:1.9}
holds, then $d$ is partial distance-preserving and
\begin{equation}\label{2:1.11}
d((x_1,y_1),(x_2,y_2))\leq2d((x_3,y_3),(x_4,y_4))
\end{equation}
whenever inequalities \eqref{2:1.2} hold.
\end{proposition}

\begin{figure}[hbt]\centering
\begin{tabular}{c|c|c|c|c|c|c|c|c|c}
$\phantom{\sum\limits_0 x_1}$ & $x_1\ y_1$ & $x_1\ y_2$ & $x_1\
y_3$ & $x_2\ y_1$ & $x_2\ y_2$ & $x_2\ y_3$
& $x_3\ y_1$ & $x_3\ y_2$ & $x_3\ y_3$\\
\hline $\phantom{\int\limits_0^1}\genfrac{}{}{0pt}{0}{x_1}{y_1}$ &
{\bf0}&1&2&{\bf1}&1&2&{\bf2}&2&2\\
\hline $\phantom{\int\limits_0^1}\genfrac{}{}{0pt}{0}{x_1}{y_2}$ &
1&{\bf0}&1&1&{\bf1}&2&2&{\bf2}&2\\\hline
$\phantom{\int\limits_0^1}\genfrac{}{}{0pt}{0}{x_1}{y_3}$ &
2&1&{\bf0}&2&2&{\bf1}&2&2&{\bf2}\\\hline
$\phantom{\int\limits_0^1}\genfrac{}{}{0pt}{0}{x_2}{y_1}$ &
{\bf1}&1&2&{\bf0}&1&2&{\bf1}&2&2\\\hline
$\phantom{\int\limits_0^1}\genfrac{}{}{0pt}{0}{x_2}{y_2}$ &
1&{\bf1}&2&1&{\bf0}&1&2&{\bf1}&2\\\hline
$\phantom{\int\limits_0^1}\genfrac{}{}{0pt}{0}{x_2}{y_3}$ &
2&2&{\bf1}&2&1&{\bf0}&2&2&{\bf1}\\\hline
$\phantom{\int\limits_0^1}\genfrac{}{}{0pt}{0}{x_3}{y_1}$ &
{\bf2}&2&2&{\bf1}&2&2&{\bf0}&1&2\\\hline
$\phantom{\int\limits_0^1}\genfrac{}{}{0pt}{0}{x_3}{y_2}$ &
2&{\bf2}&2&2&{\bf1}&2&1&{\bf0}&1\\\hline
$\phantom{\int\limits_0^1}\genfrac{}{}{0pt}{0}{x_3}{y_3}$ &
2&2&{\bf2}&2&2&{\bf1}&2&1&{\bf0}\\
\end{tabular}
\caption{The distance-matrix of a space $(X\times Y,d)$ for
$X=\{x_1,x_2,x_3\}$ and $Y=\{y_1,y_2,y_3\}$.}
\end{figure}

\begin{proof}
The first part of the proposition directly follows from
\eqref{2:1.9}, because $d_\infty$ and $d_1$ is partial
distance-preserving. To prove the second part we may use
 the following elementary inequality
$$
a+b\leq2\max\{a,b\}
$$
which holds for all $a,b\in\mathbb R$.
\end{proof}

\begin{example}\label{e2.7}
Let $(X,d_X)$ and $(Y,d_Y)$ be two three-point metric spaces such
that
$$
d_X(x_i,x_j)=d_Y(y_i,y_j)=|i-j|
$$
for all $x_i,x_j\in X$ and all $y_i,y_j\in Y$, $i,j\in\{1,2,3\}$.
Consider the metric space $(X\times Y,d)$ for which the metric $d$
is defined by the distance-matrix from Fig.~1. Then double
inequality \eqref{2:1.9} holds for all $(x_i,y_i)\in X\times Y,\
i=1,2$, and moreover we have the equalities
$$
1=d((x_1,y_1),(x_2,y_2))=\frac12d((x_2,y_2),(x_3,y_3)).
$$
Consequently $d$ is not distance-increasing and $2$ is the best
possible constant in inequality \eqref{2:1.11}.
\end{example}
The product space $(X\times Y,d)$ inherits many useful properties
of the factors if $d$ is distance-increasing and partial
distance-preserving. Recall that a metric space $(X,d)$ is {\it
proper} if each closed and bounded set $A\subseteq X$ is compact.
\begin{proposition}\label{2:p1.7}
Let $(X,d_X)$ and $(Y,d_Y)$ be two metric spaces. If $d$ is a
metric on $X\times Y$ such that \eqref{2:1.9} holds, then the
following statements are true.
\begin{itemize}

\item[$(i)$] $(X\times Y,d)$ is bounded if and only if $(X,d_X)$
and $(Y,d_Y)$ are bounded.

\item[$(ii)$] $(X\times Y,d)$ is complete if and only if $(X,d_X)$
and $(Y,d_Y)$ are complete.

\item[$(iii)$] $(X\times Y,d)$ is proper if and only if $(X,d_X)$
and $(Y,d_Y)$ are proper.

\end{itemize}
\end{proposition}
\begin{proof}
Propositions $(i)$ and $(ii)$ can be obtained by the standard
arguments.

For the proof of  $(iii)$ observe that a metric space $(Z,\rho)$
is proper if and only if every closed ball
$$
B_\rho(a,r):=\{x\in Z:\rho(x,a)\leq r\}
$$
is compact. Suppose that $(X,d_X)$ and $(Y,d_Y)$ are proper. From
the first inequality in \eqref{2:1.9} we obtain
$$
B_d((x_1,y_1),r)\subseteq
B_{d_\infty}((x_1,y_1),r)=B_{d_X}(x_1,r)\times B_{d_Y}(y_1,r).
$$
The last direct product is compact because the balls
$B_{d_X}(x_1,r)$ and $B_{d_Y}(y_1,r)$ are compact. Hence
$B_d((x_1,y_1),r)$ is compact as a closed subset of a compact set.

Suppose that $(X\times Y,d)$ is proper. By Proposition
\eqref{2:p1.6} $d$ is partial distance-preserving. Hence for every
closed ball $B_d((x_1,y_1),r)$ the sets
\begin{equation}\label{2:1.12}
(X\times\{y_1\})\cap B_d((x_1,y_1),r)\quad \text{and}\quad
(\{x_1\}\times Y)\cap B_d((x_1,y_1),r)
\end{equation}
are isometric to the balls $B_{d_X}(x_1,r)$ and $B_{d_Y}(y_1,r)$,
respectively. Since sets $X\times\{y_1\}$ and $\{x_1\}\times Y$
are closed, the sets in \eqref{2:1.12}, and hence the closed balls
$B_{d_X}(x_1,r)$ and $B_{d_Y}(y_1,r)$, are compact.
\end{proof}
\begin{theorem}\label{t2.8}
Let $(X,d_X)$ and $(Y,d_Y)$ be metric spaces and let $d$ be a
partial distance-preserving metric on $X\times Y$ such that the
inequality
\begin{equation}\label{2.14}
d_\infty((x_1,y_1),(x_2,y_2))\leq d((x_1,y_1),(x_2,y_2))
\end{equation}
holds for all $(x_i,y_i)\in X\times Y,\ i=1,2$. Then $d$ is an
ultrametric if and only if $d_X$ and $d_Y$ are ultrametrics and
$d=d_\infty$.
\end{theorem}
\begin{proof}
Suppose that $d_X$ and $d_Y$ are ultrametrics. Then for all
$(x_i,y_i)\in X\times Y,\ i=1,2,3,$ we obtain
\begin{multline*}
\max\{d_\infty((x_1,y_1),(x_2,y_2)),d_\infty((x_2,y_2),(x_3,y_3))\}\\
=\max\{\max\{d_X(x_1,x_2),d_Y(y_1,y_2)\},\max\{d_X(x_2,x_3),d_Y(y_2,y_3)\}\}\\=
\max\{\max\{d_X(x_1,x_2),d_X(x_2,x_3)\},\max\{d_Y(y_1,y_2)d_Y(y_2,y_3)\}\}\\\geq
\max\{d_X(x_1,x_3),d_Y(y_1,y_3)\}=d_\infty((x_1,y_1),(x_3,y_3)),
\end{multline*}
i.e., $(X\times Y,d_\infty)$ is an ultrametric space if $(X,d_X)$
and $(Y,d_Y)$ are ultrametric.

Conversely, let $(X\times Y,d)$ be an ultrametric space. Since $d$
is partial distance-preserving  we have
\begin{multline*}
d_X(x_1,x_3)=d((x_1,y),(x_3,y))\\\leq\max\{d((x_1,y),(x_2,y)),d((x_2,y),(x_3,y))\}\\
=\max\{d_X(x_1,x_2),d_X(x_2,x_3)\}
\end{multline*}for every $y\in Y$ and
$x_1,x_2,x_3\in X$.
 Hence $d_X$
is an ultrametric. A similar argument yields that $d_Y$ is an
ultrametric if $d$ is an ultrametric. To prove that $d=d_\infty$
it is sufficient to show that the inequality
\begin{equation}\label{2.16}
d((x_1,y_1),(x_2,y_2))\leq d_\infty((x_1,y_1),(x_2,y_2))
\end{equation}
holds for all $(x_1,y_1),(x_2,y_2)\in X\times Y$. Since
 $d$ is a partial distance-preserving ultrametric, we have
\begin{multline*}
d((x_1,y_1),(x_2,y_2))\\\leq
\max\{d((x_1,y_1),(x_1,y_2)),d((x_1,y_2),(x_2,y_2))\}\\=
\max\{d_Y(y_1,y_2),d_X(x_1,x_2)\},
\end{multline*}
i.e., \eqref{2.16} holds.
\end{proof}
\begin{remark}\label{r:3.10}
Let $(X,d_X)$, $(Y,d_Y)$ and $(X\times Y,d)$ be metric spaces such
that $d_\infty\leq d$. It follows from Proposition \ref{2:p1.6}
and inequality \eqref{3:12*} that a metric $d$ is partial
distance-preserving if and only if $d\leq d_1$.
\end{remark}
\begin{corollary}\label{c2.9}
Let $(X,d_X)$ and $(Y,d_Y)$ be metric spaces and let $d$ be a
distance-increas\-ing and partial distance-preserving metric on
the product $X\times Y$. Then $d$ is an ultrametric if and only if
$d_X$ and $d_Y$ are ultrametrics and $d=d_\infty$.
\end{corollary}
\begin{proof}
It follows from Theorem \ref{t2.8}  and Proposition \ref{2:p1.4}.
\end{proof}

\section{Products of   packing numbers and  \\products of
ultrametric spaces}

In this section we give some  conditions under which a product of
metric spaces is ultrametric.
\begin{theorem}\label{t:3:1}
Let $(X,d_X)$ and $(Y,d_Y)$ be ultrametric spaces and let $d$ be a
partial distance-preserving metric on $(X\times Y)$ such that the
inequality
\begin{equation}\label{3:1}
d_\infty((x_1,y_1),(x_2,y_2))\leq d((x_1,y_1),(x_2,y_2))
\end{equation}
holds for all $((x_1,y_1),(x_2,y_2))\in (X\times Y)\times(X\times
Y)$. Then the following statements are equivalent.
\begin{itemize}
\item[$(i)$] $d$ is an ultrametric on $X\times Y$.

\item[$(ii)$] The equality
\begin{equation}\label{3:2}
\mathcal M_\varepsilon(W\times Z)=\mathcal M_\varepsilon(W)\cdot
\mathcal M_\varepsilon(Z)
\end{equation}
holds for all compact sets $W\subseteq X$ and $Z\subseteq Y$ and
every $\varepsilon>0$.
\end{itemize}
\end{theorem}
\begin{lemma}\label{l:3:2}
Let $(X,d_X)$, $(Y,d_Y)$ and $(X\times Y,d)$ be ultrametric
spaces. Suppose that $d$ is partial distance-preserving and
\eqref{3:1} holds for all $((x_1,y_1),(x_2,y_2))\in(X\times
Y)\times(X\times Y)$. Then the equalities \eqref{3:2},
\begin{equation}\label{3:3}
\mathcal N_\varepsilon(W\times Z)=\mathcal N_\varepsilon(W)\cdot
\mathcal N_\varepsilon(Z),
\end{equation}
and
\begin{equation}\label{3:4}
\mathcal M_\varepsilon(W\times Z)=\mathcal N_\varepsilon(W\times
Z)
\end{equation}
hold for all compact sets $W\subseteq X$, $Z\subseteq Y$ and every
$\varepsilon>0$.
\end{lemma}
\begin{proof}
Let $W$ and $Z$ be compact sets $W\subseteq X,\ Z\subseteq Y$ and
let $\varepsilon>0$. Theorem \ref{t2.8} implies that $d=d_\infty$
if the conditions of the lemma hold. It follows from the
definition of the covering numbers that
\begin{equation}\label{3:5}
\mathcal N_\varepsilon(W\times Z)\leq\mathcal
N_\varepsilon(W)\cdot\mathcal N_\varepsilon(Z).
\end{equation}
Indeed, if $C_W$ and $C_Z$ are finite $\varepsilon$-nets for $W$
and, respectively, for $Z$, then the direct product $C_W\times
C_Z$ is a finite $\varepsilon$-net for $W\times Z$ in the space
$(X\times Y,d_\infty)$. Consequently, we obtain
$$
\mathcal N_\varepsilon(W\times Z)\leq \card(C_W)\cdot\card(C_Z).
$$
Using this inequality for $C_W$ and $C_Z$ with
$\card(C_W)=\mathcal N_\varepsilon(W)$ and $\card(C_Z)=\mathcal
N_\varepsilon(Z)$ we obtain \eqref{3:5}. Similarly, the definition
of the packing numbers implies  the inequality
\begin{equation}\label{3:6}
\mathcal M_\varepsilon(W\times Z)\geq\mathcal
M_\varepsilon(W)\cdot\mathcal M_\varepsilon(Z).
\end{equation}
for the subspace $W\times Z$ of the space $(X\times Y,d_\infty)$.

Statement $(iii)$ of Theorem \ref{t:1.6} gives
$$
\mathcal M_\varepsilon(W)=\mathcal N_\varepsilon(W)\quad
\text{and}\quad \mathcal M_\varepsilon(Z)=\mathcal
N_\varepsilon(Z)
$$
because $(X,d_X)$ and $(Y,d_Y)$ are ultrametric spaces. The metric
$d=d_\infty$ induces the natural topology on $X\times Y$. Thus
$W\times Z$ is compact in $(X\times Y,d_\infty)$, so  Theorem
\ref{t:1.6} $(iii)$ implies also the equality
$$
\mathcal M_\varepsilon(W\times Z)=\mathcal N_\varepsilon(W\times
Z).
$$
Consequently, from \eqref{3:5} and \eqref{3:6} we obtain
\begin{multline*}
\mathcal N_\varepsilon(W)\mathcal N_\varepsilon(Z)=\mathcal
M_\varepsilon(W)\mathcal M_\varepsilon(Z)\leq\mathcal
M_\varepsilon(W\times Z)\\=\mathcal N_\varepsilon(W\times
Z)\leq\mathcal N_\varepsilon(W)\mathcal N_\varepsilon(Z).
\end{multline*}
Equalities \eqref{3:2}--\eqref{3:4} are proved.
\end{proof}
\begin{proof}[\it Proof of Theorem \ref{t:3:1}.]
It was shown in Lemma \ref{l:3:2} that $(i)\Rightarrow(ii)$. To
prove  $(ii)\Rightarrow(i)$ suppose that \eqref{3:2} holds for
every $\varepsilon>0$ and all compacts $W\subseteq X,\ Z\subset Y$
but $(X\times Y,d)$ is not ultrametric. Then, by Theorem
\ref{t2.8}, there are points $(x_i,y_i)\in X\times Y,\ i=1,2,$
such that
\begin{equation}\label{3:7}
\max\{d_X(x_1,x_2),d_Y(y_1,y_2)\}<d((x_1,y_1),(x_2,y_2)).
\end{equation}
Write
\begin{equation}\label{3:8}
W:=\{x_1,x_2\},\qquad Z:=\{y_1,y_2\}
\end{equation}
and
\begin{equation}\label{3:9}
\varepsilon:=\max\{d_X(x_1,x_2),d_Y(y_1,y_2)\}.
\end{equation}
Then we evidently have
\begin{equation}\label{3:10}
\mathcal M_\varepsilon(W)=\mathcal M_\varepsilon(Z)=1.
\end{equation}
Note also that inequality \eqref{3:7} implies that the set
$\{(x_1,y_1),(x_2,y_2)\}$ is an $\varepsilon$-distinguish\-able
subset of $W\times Z$ in the space $(X\times Y,d)$. Hence, we have
the inequality $\mathcal M_\varepsilon(W\times Z)\geq2$. This
inequality and \eqref{3:10} contradict \eqref{3:2}. Hence, the
implication $(ii)\Rightarrow(i)$ holds.
\end{proof}
If $d$ is partial distance-preserving and $d_\infty\leq d$ but
only one from the spaces $(X,d_X)$ and $(Y,d_Y)$ is ultrametric,
then, generally, the metric space $(X\times Y,d)$ may be
nonultrametric even if \eqref{3:2} holds for all compact sets
$W\subseteq X,\ Z\subseteq Y$ and every $\varepsilon>0$.
\begin{example}\label{e3:3}
Let $X=\{x\}$ be an one-point metric space. Then $X$ is
ultrametric and for every $(Y,d_Y)$ there is a unique partial
distance-preserving metric $d=d_\infty$ on $X\times Y$, i.e., the
function
$$
(X\times Y,d)\ni(x,y)\longmapsto y\in(Y,d_Y)
$$
is an isometry if $d$ is partial distance-preserving. Furthermore,
it is clear that every $W\subseteq X$ is either empty or one-point
and
$$
\mathcal M_\varepsilon(\emptyset)=\mathcal M_\varepsilon(X)-1=0.
$$
Hence \eqref{3:2} holds for all compact sets $W\subseteq X,\
Z\subseteq Y$ and every $\varepsilon>0$ but $(X\times Y,d)$ is
ultrametric if and only if $(Y,d_Y)$ is ultrametric.
\end{example}
\begin{proposition}\label{p:3:4}
Let $(X,d_X)$ and $(Y,d_Y)$ be metric spaces and let $d$ be a
partial distance-preserving metric on $X\times Y$ such that
$d_\infty\leq d$. Then the space $(X\times Y,d)$ is ultrametric if
and only if the equalities
\begin{equation}\label{3:11}
\mathcal N_\varepsilon(W\times Z)=\mathcal M_\varepsilon(W\times
Z)
\end{equation}
and \eqref{3:2} hold for all compact sets $W\subseteq X,\
Z\subseteq Y$ and every $\varepsilon>0$.
\end{proposition}
The following fact is included in the proof of Theorem \ref{t2.8}.
\begin{lemma}\label{l:3:5}
Let $(X,d_X)$ and $(Y,d_Y)$ be metric spaces and let $d$ be a
partial distance-preserving ultrametric on $X\times Y$. Then
$(X,d_X)$ and $(Y,d_Y)$ are ultrametric spaces.
\end{lemma}

\begin{proof}[\it Proof of Proposition \ref{p:3:4}.] If $(X\times
Y,d)$ is ultrametric, then, by Lem\-ma~\ref{l:3:5}, $(X,d_X)$ and
$(Y,d_Y)$ are ultrametric. Consequently, \eqref{3:2} follows from
Theorem \ref{t:3:1}. The set $W\times Z$ is compact if $W$ and $Z$
are compact. Hence, \eqref{3:11}  follows from Theorem \ref{t:1.6}
$(iii)$.

Now suppose that \eqref{3:11} and \eqref{3:2} hold for all compact
sets $W\subseteq Z$, $Z\subseteq Y$ and every $\varepsilon>0$. To
prove that $(X\times Y,d)$ is ultrametric, it is sufficient, by
Theorem \ref{t:3:1}, to show that $(X,d_X)$ and $(Y,d_Y)$ are
ultrametric spaces. Using \eqref{3:11} with an one-point set $W$
we see that $ \mathcal N_\varepsilon(Z)=\mathcal M_\varepsilon(Z)
$ for every compact set $Z\subseteq Y$ and every $\varepsilon>0$,
because $d$ is partial distance-preserving. Hence, by Theorem
\ref{t:1.6}, $Y$ is an ultrametric space. Similarly $X$ is an
ultrametric space.
\end{proof}
The following example shows that  in Theorem \ref{t:3:1} the
packing numbers cannot be replaced by covering numbers.

\begin{figure}[thb]\centering
\begin{tabular}{c|c|c|c|c}
$\phantom{x_1\ }$&$x_1\ y_1$& $x_1\ y_2$& $x_2\ y_1$ &$x_2\
y_2$\\\hline
$\phantom{\int\limits_0^1}\genfrac{}{}{0pt}{0}{x_1}{y_1}$ &
0&1&1&$a$\\\hline
$\phantom{\int\limits_0^1}\genfrac{}{}{0pt}{0}{x_1}{y_2}$
&1&0&1&1\\\hline
$\phantom{\int\limits_0^1}\genfrac{}{}{0pt}{0}{x_2}{y_1}$ &
1&1&0&1\\\hline
$\phantom{\int\limits_0^1}\genfrac{}{}{0pt}{0}{x_2}{y_2}$
&$a$&1&1&0\\
\end{tabular}
\caption{The distance-matrix of a metric space $(X\times Y,d)$ for
$X=\{x_1,x_2\}$ and $Y=\{y_1,y_2\}$. Here $a$ is an arbitrary real
number from $[1,2]$.}
\end{figure}

\begin{example}\label{e:3:6}
Let $X=\{x_1,x_2\}$ and $Y=\{y_1,y_2\}$ be two-point metric spaces
with  metrics $d_X,d_Y$ such that
$$
d_X(x_1,x_2)=d_Y(y_1,y_2)=1.
$$
Let $(X\times Y,d)$ be a product of the spaces $(X,d_X)$ and
$(Y,d_Y)$  such that $d$ is generated by the distance-matrix from
Fig.~2. Then $d$ is a partial distance-preserving and
$d_\infty\leq d$. Moreover, a computation shows that \eqref{3:3}
holds for all $W\subseteq X$, $Z\subseteq Y$ and every
$\varepsilon>0$. Specifically we have
$$
\mathcal N_1(X\times Y)=\mathcal N_1(X)\cdot\mathcal N_1(Y)
$$
because
$$
B_d((x_1,y_2),1)\supseteq X\times Y.
$$

Note that $(X\times Y,d)$ is not an ultrametric space if
$1<a\leq2$.
\end{example}
\begin{proposition}\label{p:3:7}
Let $(X,d_X)$ and $(Y,d_Y)$ be ultrametric spaces and let $d$ be a
partial distance-preserving metric on $X\times Y$ such that
$d_\infty\leq d$. Suppose that  \eqref{3:3} holds for all compact
sets $W\subseteq Z,\ Z\subseteq Y$ and every $\varepsilon>0$. Then
\begin{equation}\label{3:12}
\min\{d((x_1,y_1),(x_2,y_2)),d((x_2,y_1),(x_1,y_2))\}=d_\infty\{(x_1,y_1),(x_2,y_2)\}
\end{equation}
holds for all $\{x_1,x_2\}\subseteq X$ and $\{y_1,y_2\}\subseteq
Y$.
\end{proposition}
\begin{proof}
Suppose that \eqref{3:12} does not hold for some $x_1,x_2\in X$
and $y_1,y_2\in Y$. Then using the inequality $d_\infty\leq d$ we
see that
\begin{equation}\label{3:13}
d((x_1,y_1),(x_2,y_2))>\max\{d_X(x_1,x_2),d_Y(y_1,y_2)\}
\end{equation}
and
\begin{equation}\label{3:14}
d((x_2,y_1),(x_1,y_2))>\max\{d_X(x_1,x_2),d_Y(y_1,y_2)\}.
\end{equation}
Write
$$
W:=\{x_1,x_2\},\qquad Z:=\{y_1,y_2\},\qquad
\varepsilon:=d_\infty((x_1,x_2),(y_1,y_2)).
$$
Then it is clear that
\begin{equation}\label{3:15}
\mathcal N_\varepsilon(W)=\mathcal N_\varepsilon(Z)=1.
\end{equation}
Moreover, inequalities \eqref{3:13} and \eqref{3:14} imply that
$$
(W\times Z)\setminus B_d((x,y),\varepsilon)\ne\emptyset
$$
for every $(x,y)\in X\times Y$. Consequently, we have
$
\mathcal N_\varepsilon(W\times Z)>1.
$
To complete the proof, it suffices to observe that the last
inequality and \eqref{3:15} contradict \eqref{3:3}.
\end{proof}
\begin{corollary}\label{c:3:8}
Let $(X,d_X)$ and $(Y,d_Y)$ be ultrametric spaces and let $d$ be a
partial distan\-ce-preserving metric on $X\times Y$ such that
$d_\infty\leq d$. Suppose that the equality
\begin{equation}\label{3:16}
d((x_1,y_1),(x_2,y_2))=d((x_2,y_1),(x_1,y_2))
\end{equation}
holds for all $\{x_1,x_2\}\subseteq X$ and all
$\{y_1,y_2\}\subseteq Y$. Then $(X\times Y,d)$ is ultrametric if
and only if \eqref{3:3} holds for all compact sets $W\subseteq X$,
$Z\subseteq Y$ and every $\varepsilon>0$.
\end{corollary}
\begin{proof}
Suppose that $(X\times Y,d)$ is ultrametric. Then \eqref{3:3}
holds, see Lemma \ref{l:3:2}. Conversely, if \eqref{3:3} holds for
all compact $W\subseteq X$, $Z\subseteq Y$ and every
$\varepsilon>0$, then, by Proposition~\ref{p:3:7}, we have
 \eqref{3:12}, Note that \eqref{3:12} and \eqref{3:16}
imply the equality $d=d_\infty$. Using Theorem \ref{t2.8} we see
that $(X\times Y,d)$ is an ultrametric space.
\end{proof}
\begin{remark}\label{r:3:9}
If the distance function $d:(X\times Y)\times(X\times Y)\to\mathbb
R^+$ depends only on ``partial'' distances $d_X$ and $d_Y$, see
diagram \eqref{2:1.6}, then  \eqref{3:16} evidently holds. Note
that \eqref{3:16} holds for all points from the space $(X\times
Y,d)$ in Example \ref{e2.7} but, in this case, there is no
function $F$ for which diagram \eqref{2:1.6} is commutative.
\end{remark}

{\bf Acknowledgment.} The first author thanks the Department of
Mathematics and Statistics of the University of Helsinki and the
Academy of Finland for the support.

\end{document}